\documentclass{imanum}
\usepackage{graphicx}

\jno{drnxxx}
\received{\today}
\revised{}

\usepackage{amsmath,amssymb,amsthm}
\usepackage{graphicx}
\usepackage{soul}
\usepackage{color}
\usepackage{ulem}
\usepackage{epstopdf}
\usepackage{booktabs} 

\def\alb{al\-ge\-braic}

\def\apl{application}
\def\apr{approximat}

\def\bc{boundary condition}
\def\bd{{\bf d}}

\def\bvp{boundary value problem}

\def\coe{coefficient}
\def\corr{corresponding}

\def\dim{di\-men\-sion}
\def\dimn{d}

\def\disc{discretiz}
\def\disp{\displaystyle}

\def\ell{elliptic}
\def\eq{equation}

\def\er{{\mathbb{R}}}

\def\fe{finite element}
\def\fem{finite element method}
\def\gad{\partial\S} 

\def\gh{g_h}
\def\hed{W^{1,p}_0(\Om)}

\def\K{{\bf K}}

\def\lpv{L^{p'}(\Om)}

\def\li{linear}

\def\M{{\cal M}}
\DeclareMathOperator{\meas}{meas}

\def\nb{\nabla}

\def\nm{numerical}

\def\Om{\Omega}
\def\op{operator}
\def\ov{\overline}

\def\pa{\partial}

\def\pr{problem}

\def\pcs{p^\ast}

\def\resp{respectively}

\def\sl{solution}

\def\sq{sequence}

\def\sth{such that}
\def\sy{system}

\def\uh{{u_h}}

\def\vh{V_h}
\def\vhh{v_h}
\def\vhn{V_h^0}

\def\wip{W^{1,p}(\S)}
\def\wep{W^{1,p}_0(\S)}
\def\ws{weak solution}

\def\sm{\smallskip}
\def\bi{\bigskip}
\def\me{\medskip}

\def\bi{\bigskip}
\def\sm{\smallskip}
\def\me{\medskip}

\DeclareMathOperator{\dint}{\mathrm{d}\!}
\renewcommand{\d}{\dint}
\DeclareMathOperator{\diverg}{div}
\renewcommand{\div}{\diverg}

\newcommand{\R}{\mathbb{R}}
\DeclareMathOperator{\spn}{span}
\DeclareMathOperator{\supp}{supp}
\newcommand{\Tcalh}{\mathcal{T}_h}

\renewcommand{\S}{S}                
\newcommand{\nbs}{\nabla_{S}}       
\newcommand{\nbsh}{\nabla_{S_h}}    

\newcommand{\hatp}{\hat p}

\begin{document}

\title{Discrete maximum principles for nonlinear elliptic finite element problems on Riemannian manifolds with boundary}
\shorttitle{DMPs for nonlinear elliptic finite element problems on Riemannian manifolds}

\author{%
{\sc
J\'anos Kar\'atson\thanks{Email: karatson@cs.elte.hu},} \\[2pt]
Department of Applied Analysis \& MTA-ELTE Numerical Analysis and Large Networks
Research Group, ELTE University;\\
Department of   Analysis, Technical University; Budapest, Hungary;\\
{\sc and}\\[6pt]
{\sc Bal\'azs Kov\'acs}\thanks{Corresponding author. Email: kovacs@na.uni-tuebingen.de}\\[2pt]
Mathematisches Institut, Universit\"at T\"ubingen,\\
Auf der Morgenstelle 10, 72076 T\"ubingen, Germany\\
{\sc and}\\[6pt]
{\sc Sergey Korotov}\thanks{Email: sergey.korotov@hvl.no}\\[2pt]
Department of Computing, Mathematics and Physics, \\
Western Norway University of Applied Sciences, Bergen, Norway
}
\shortauthorlist{J.~Kar\'atson, B.~Kov\'acs and S.~Korotov}

\maketitle

\begin{abstract}
{The maximum principle forms an important quali\-ta\-ti\-ve property of second order elliptic \eq s, therefore its discrete analogues, the so-called discrete maximum principles (DMPs) have drawn much attention owing to their role in reinforcing the qualitative reliability of the given numerical scheme.
In this paper  DMPs are established for nonlinear surface finite element problems on Riemannian manifolds, corresponding to the classical pointwise maximum principles on surfaces in the spirit of Pucci et al. Various real-life examples illustrate the scope of the results.}
{nonlinear elliptic problems on surfaces; maximum principles; discrete maximum principles; surface finite element method; simplicial mesh.}
%
\end{abstract}

\section{Introduction}

The maximum principle forms an important quali\-ta\-ti\-ve property of second order elliptic \eq s \cite{GT, PrW, Pucci_Serrin}. We will refer to such results as continuous  maximum principles (CMPs). Typical maximum principles often arise   in a form  stating that a solution attains a nonnegative maximum on the boundary, i.e.
\begin{equation}
\label{dmpintr}
\max\limits_{\ov\S} u \le \max \{ 0, \max\limits_{\gad} u \}.
\end{equation}
CMPs have also been studied on Riemannian manifolds. Thereby the main subject is often an Omori--Yau type maximum principle, but in this paper we are interested in classical pointwise versions, such as the ones studied by Antonini, Mugnai and Pucci in \cite{Pucci}, which give analogous results to (\ref{dmpintr}).

\sm

The discrete analogues of CMPs, the so-called discrete maximum principles (DMPs) have drawn much attention. The DMP is in fact an important measure of the qualitative reliability of the \nm \ scheme, otherwise one could get unphysical \nm \ \sl s like negative concentrations, etc.

We are interested for DMPs on a surface with boundary, in the context of the surface finite element method (SFEM). Originally, the surface FEM was developed by Dziuk in \cite{Dziuk88} for elliptic problems on closed surfaces (and further developed for evolving surfaces by Dziuk and Elliott in \cite{DziukElliott_ESFEM,DziukElliott_acta}). The theory was extended to elliptic problems on surfaces with boundaries in the preprint \cite{boundarySFEM}. 
In such a case a DMP reproduces  the above relation  for the surface finite element solution $u_h$ instead of $u$.
Various DMPs in Euclidean domains, including geometric conditions on the computational meshes for FEM solutions, have been given, e.g.\ in \cite{Ci70, CiRav, Failure, XuZik}. The authors' previous work, e.g.\ \cite{KarKor-NM, KK, KKN}, involves various types of linear and nonlinear equations and systems, and typical geometric conditions are nonobtuseness or acuteness in the case of simplicial meshes.

\sm

To our knowledge, no DMP has been established for (even linear) elliptic problems on surfaces yet. In this paper we provide such  results in the spirit of \cite{Pucci} for nonlinear  finite element problems on manifolds with boundary. We suitably adapt a 
treatment of the \fe \ matrix from  our earlier paper on the flat domain case  \cite{KarKor-NM}. We work in the mentioned
surface finite element setting, cf.\ \cite{boundarySFEM, Dziuk88}, and use the properties of the lift operator to achieve the arising stability requirements, thus
various discrete  maximum-minimum principles are established.
We note that discrete maximum principles for parabolic reaction--diffusion 
equations on closed surfaces have been treated in \cite{FMSV1,FMSV2,FMSV3}.

In this paper the issue of generating meshes with
suitable angle properties is also addressed. The theoretical results are supported by
numerical tests.

\section{Formulation and properties of the problem}

Our formulation  of the problem is based on \cite{Pucci}.  We consider a non\li \ \ell \ \pr \ which 
belongs to the class of  \pr s studied in \cite{Pucci} on manifolds.
Our goal is to adapt our approach from \cite{KarKor-NM} developed for Euclidean domains. We first show that a proper CMP can be derived, and then we verify  DMPs   for proper SFEM \disc ations.

We follow \cite{Pucci} and \cite{boundarySFEM, Dziuk88} for the  description of the \ell \ \pr \ on the manifold and of the SFEM, \resp .

\subsection{The nonlinear elliptic problem on a Riemannian manifold}

\subsubsection{The manifold}
\label{subsubmf}

Following the setting and notations of  \cite{Pucci}, let $\M$ be a smooth complete  $d$-\dim al Riemannian  manifold, and let $\S \subset \M$ be a closed bounded and regular subdomain of $\M$, so that $\ov\S$­ is a smooth
manifold with boundary, with outward normal $\nu$. 
In the sequel we shall write $T\S\times_{\S} \er$ in place of $T\S \times_{\S} (\S\times\er)$ to denote the fibered product bundle. In analogy with the Euclidean case, points of  $T\S\times_{\S} \er$ will be denoted with $(x,z,\xi)$, where $(x, \xi)\in T\S$ and $(x, z)\in \S\times\er$. Integrals will be taken with respect to the natural Riemannian measure, denoted by $d\M$.
The surface gradient, or tangential gradient (i.e.\ the orthographic projection of the standard gradient onto the surface) of a function $u$ on $\S$ will be simply denoted by $\nbs u$, i.e. $\nbs u = \nb u - (\nb u \cdot \nu) \nu$. Further, $\Delta_{\S}$ denotes the Laplace--Beltrami operator, and $\div_\S$ the tangential divergence.

For $p\ge 1$, the Lebesgue and Sobolev norms are defined as
$$
\|u\|_{p,\S}:= \Bigl( \int_{\S} |u|^p \d\M   \Bigr)^{1/p} \quad  \mbox{and} \qquad \|u\|_{1,p,\S}:=  \|u\|_{p,\S} +  \|\nb u\|_{p,\S}\, ,
$$
\resp , where $|u|:=g(u,u)^{1/2}$ using the metric tensor $g$, further, $\wip$ and $\wep$ are defined as the closure of  $C^\infty(\S)$ and  $C_0^\infty(\S)$, \resp , in the Sobolev norm $\|\cdot \|_{1,p,\S}$. We define $p':= p/(p-1)$ if $p>1$ and $p':=\infty$ if $p=1$ (see e.g.\ \cite{Dziuk88,DziukElliott_acta}).

There are similar Sobolev embedding estimates of $\wip$ available as in the Euclidean case. Let $p^\ast:=
\frac{p\dimn}{\dimn-p}$  if $\dimn>p$ and $p^\ast:=\infty$ if $\dimn\le p$. Here $\S$ is a closed bounded subdomain of the complete Riemannian $d$-manifold $\M$, hence it is a compact Riemannian $d$-manifold itself. Therefore, by \cite{Au}, for all $1\le p_1<\pcs$  we have $W^{1,p}(\S) \subset L^{p_1}(\S)$ and
\begin{equation}
\label{Sobombas}
\|v\|_{p_1,\S}  \le k_1 \|v\|_{1,p,\S}
\qquad (\forall v\in\wip)
\end{equation}
for some constant  $k_1>0$ depending on $p$ and $p_1$ but independent of $v$.

\subsubsection{The \pr }
\label{subsubpr}

We consider a nonlinear boundary value problem of the following type, involving a scalar non\li \ \coe \ in the principal part:
\begin{equation}
\label{mixbvpex}
\left\{
\begin{aligned}
-\div_{\S} \Bigl(b(x,u,\nbs u)\, \nbs u\Bigr) +q(x,u) =&\ f(x) &\qquad& {\rm on } \, \S,
\\
u   =&\  g(x)   &\qquad &  {\rm on} \
\partial\S,
\end{aligned}
\right.
\end{equation}
with the bounded surface $\S\subset\M$, with boundary $\pa\S$, under the following
conditions:

\bi
\noindent {\bf Assumptions \ref{subsubpr}.}
\begin{enumerate}
	\item[(A1)]
	Let $p\ge 2$ be a given parameter.
	
	\item[(A2)] The scalar functions $b:\S \times \R \times \R^{d+1} \to \er$, \, $q: \ov \S \times \er \to\er$
	are continuous.
	There holds
	$f\in \lpv$, further, 
	$g=g^*|_{\pa\S}$ for some $g^*\in W^{1,\gamma}(\S)$, where  $\gamma\ge p$ if $p > d$ and $\gamma> d$    if $p\leq d$ .
	\item[(A3)] (Uniform ellipticity.)
	The function $b$ satisfies
	\begin{equation*}
	\mu_0 + \mu_1 |\xi|^{p-2} \leq b(x,z,\xi) \leq M_0 + M_1 |\xi|^{p-2}
	\end{equation*}
	with constants $\mu_0, \mu_1, M_0, M_1>0$  independent of $(x,z,\xi)\in \S \times \R \times \R^{d+1}$.

	\item[(A4)]   Let $2\le p_1< +\infty$ \ if $\dimn\le p$, or $2\le p_1   < \hatp := \frac{2\dimn}{\dimn-p}$ if $\dimn>p$.
	There exist constants $\alpha, \beta\ge 0$ \sth \ for any $x\in\S$ and $z\in\er$ the following growth conditions hold: for every $z\ge 0$,
	\begin{equation*}
	0\le q(x,0)\equiv q(x,z)  \quad (\forall z\le 0), \qquad
	0\le q(x,z)- q(x,0) \leq \alpha  z + \beta z^{p_1-1} .
	\end{equation*}
\end{enumerate}

The  \ws \ $u\in\wip$ of \pr \ (\ref{mixbvpex}) will be  defined as follows:
\begin{equation}
\label{wsint}
\int\limits_{\S} \bigl[ b(x,u,\nbs u)\ \nbs u\cdot \nbs v + q(x,u)v
\bigr] \,  \d\M \,
=
\int\limits_\S fv \, \d\M \,
\qquad  (\forall v\in\wep)
\end{equation}
\begin{equation}
\label{wsbdy}
\hbox{ and}  \quad  u \, - \, g^\ast \in \wep .
\end{equation}

\begin{remark}
	\rm
	As mentioned before, related problems have been studied in the context of the DMP on usual Euclidean domains in our earlier paper \cite{KarKor-NM}, i.e.
	the above problem \eqref{mixbvpex} is a surface analogue of the problems considered therein.
	Now in \eqref{mixbvpex} we only consider Dirichlet \bc s, on the other hand, compared to \cite{KarKor-NM},
	we impose more general growth conditions allowing $p\neq 2$ as well, which involves a more general setting using the Banach space $\hed$.
\end{remark}

In the sequel we omit the sign $\d\M$ in the integrals for brevity, i.e. $\int_S f:= \int_S f\d\M$.

\subsection{Continuous maximum principle on the manifold}

\subsubsection{Preliminaries}
\label{subsubcmprel}

We will rely on \cite{Pucci}, where   CMPs are proved for a general class of \ell \  inequalities on manifolds. We summarize the required background in a rewritten form involving \ell \ \eq s with right-hand side of given sign.

\sm
Let us consider an  \ell \ \eq
\begin{equation}
\label{bvpucci}
-\div_{\S} \K(x,u,\nbs u)  +F(x,u,\nbs u)\, = \, k(x)
\end{equation}
on the domain $\S\subset\M$ under the following
conditions:

\bi

\sm

\noindent {\bf Assumptions \ref{subsubcmprel}.}
\begin{enumerate}
	\item[(i)] The functions
	$\K:T\S\times_{\S} \er\to T\S$ and
	$F: T\S\times_{\S} \er\to\er$
	are continuous, further, $\K(x, z,\xi)\in  T_x\M$ for all $(x, z)\in \S\times\er$ and $\xi\in T_x\M$.
	
	\item[(ii)]
	There exist constants
	$a_1 >0$ and $b_1,a_2\ge 0$  \sth \ for
	all $(x,z,\xi)\in T\S\times_{\S} \er$ there holds
	$$
	\K(x, z,\xi)\cdot \xi \ge a_1|\xi|^p - a_2 |z|^p, \qquad F(x, z,\xi) \ge -b_1 |\xi|^p\ .
	$$
\end{enumerate}

The treatment in \cite{Pucci} involves {\it p-regular weak solutions} of
(\ref{bvpucci}), defined as functions $u\in L^1_{loc}(\S)$ satisfying
$$
\K(\cdot,u,\nbs u)\in L^{p'}_{loc}(\S, T\S), \quad F(\cdot,u,\nbs u)\in L^{p'}_{loc}(\S)
$$
and demanding (\ref{wsint}) only for compactly supported $v\in\wep$.
The following result holds:

\begin{theorem}
	\label{thpucci}
	Let Assumptions \ref{subsubcmprel} hold  with $a_2 = 0$ in item (ii). Let $u$ be a p-regular weak solution of (\ref{bvpucci}) \sth \ $u\in W^{1,p}_{loc}(\S)$, and let the right-hand side $k\le 0$ on $\S$.
	If $u \le M$ on $\pa\S$
	for some constant $M\ge 0$, then $u\le M$  in $\S$.
\end{theorem}
In fact, this result is the reformulation of \cite[Theorem 3.3.]{Pucci} for \eq s instead of inequalities and with rearranged signs.

\subsubsection{CMP for the studied class}

We may now prove the CMP by reformulating  the above result in the vein of inequality (\ref{dmpintr}).

\begin{theorem}
	\label{thcmp}
	Let  Assumptions \ref{subsubpr} hold. Assume that $u$ is a \ws \ of \pr \ (\ref{mixbvpex}), defined as in (\ref{wsint})--(\ref{wsbdy}), which satisfies  $u\in  C(\ov\S)$. Let
	\begin{equation}
	\label{assqs}
	f(x)-q(x,0)\le 0  \quad (a.e.\ \forall x\in\S ).
	\end{equation}
	
	\sm
	(1)
	Then
	\begin{equation}
	\label{cmp1}
	\max\limits_{\ov\S} u \le \max \{ 0, \max\limits_{\pa\S} g \}.
	\end{equation}
	
	(2) In particular, if   $g\ge 0$  then
	\begin{equation}
	\label{cmp2}
	\max\limits_{\ov\S} u = \max\limits_{\pa\S} g ,
	\end{equation}
	and, if   $g\le 0$  then we have the nonpositivity property
	\begin{equation*}
	\max\limits_{\ov\S} u \le 0.
	\end{equation*}
	
	(3) If we do not assume $u\in C(\ov\S)$ but only $g\in L^\infty(\S)$, then the above statements hold with replacing each $\max $ 
	by $\rm ess \, sup$,
	\resp .
\end{theorem}

\begin{proof}
	We verify that \pr \ (\ref{mixbvpex}) under  Assumptions \ref{subsubpr} is a special case of \pr \  (\ref{bvpucci}) under Assumptions \ref{subsubcmprel}. Throughout, the notation  $(x,z,\xi)$
	will stand for the general points of $\S \times \R \times \R^{d+1}$. 
	We first subtract $q(x,0)$ from both sides of  (\ref{mixbvpex}), then
	\begin{equation}
	\label{mixbvpexsubtr}
	\left\{
	\begin{aligned}
	-\,{\rm div}_{\S}\, \Bigl(b(x,u,\nbs u)\, \nbs u\Bigr) +\hat q(x,u) =&\ \hat f(x) &\qquad &  {\rm on \, } \ \S,
	\\
	\disp   \hskip50.5mm     u =&\ g(x)   &\qquad &  {\rm on} \
	\partial\S
	\end{aligned}
	\right.
	\end{equation}
	where
	\begin{equation*}
	\hat q(x,z):=q(x,z)-q(x,0), \qquad \hat f(x):=f(x)-q(x,0)
	\end{equation*}
	for all $x\in\S$, $z\in\er$. Let
	$$
	\K(x, z,\xi):=b(x,z,\xi)\, \xi, \qquad F(x, z,\xi):=\hat q(x,z).
	$$
	Owing to the continuity of $b$ and $q$, assumption (i) obviously holds. Further, for all arguments, assumptions (A3) and (A4) yield
	$$
	\K(x, z,\xi)\cdot \xi
	= b(x,z,\xi)\, |\xi|^2
	\ge
	\mu_0  |\xi|^2 + \mu_1 |\xi|^{p}
	\ge \mu_1 |\xi|^{p}
	$$
	and
	$$
	F(x, z,\xi)\ge 0,
	$$
	i.e. assumption (ii) holds with $a_1:=\mu_1>0$, $a_2=0$ and $b_1=0$.
	
	\sm
	We also have to check that $u$ is a $p$-regular weak solution of (\ref{bvpucci}) \sth \ $u\in W^{1,p}_{loc}(\S)$, and the right-hand side satisfies $k\le 0$ on $\S$.
	Here, first, assumption (A3) yields
	$$
	|\K(x, z,\xi)|  \le b(x,z,\xi)\, |\xi| \le  M_0 |\xi| + M_1 |\xi|^{p-1} ,
	$$
	hence
	$$
	|\K(\cdot,u,\nbs u)|^{p'} \le  M_3 |\nbs u|^{p'} + M_4 |\nbs u|^{(p-1)p'} = M_3 |\nbs u|^{p'} + M_4 |\nbs u|^{p}.
	$$
	Here the second term is integrable on $\S$ since the \ws \ $u\in\wip$, and also the first term is  integrable since $p'\le 2\le p$. Hence  $\K(\cdot,u,\nbs u)\in L^{p'} (\S, T\S)$. Similarly, from (A4), for all arguments
	$$
	0\le \hat q(x,z) \le \alpha  |z| + \beta |z|^{p_1-1} ,
	$$
	hence
	$$
	|F(\cdot,u,\nbs u)|^{p'} = |q(\cdot,u) |^{p'} \le \tilde\alpha  |u|^{p'} + \tilde\beta |u|^{p}
	$$
	is integrable on $\S$ since $u\in L^p(\S)$ and $p'\le 2\le p$. Also, since  $u\in\wip$, all local requirements are fulfilled even globally on $\S$. Equation (\ref{wsint}) holds for all $v\in\wep$, hence especially for compactly supported $v\in\wep$. Altogether, $u$ satisfies all criteria of a $p$-regular weak solution. Further, for all $x\in\S$, now
	$$
	k(x):=\hat f(x) := f(x)-q(x,0) \le 0 .
	$$
	That is, we can apply Theorem~\ref{thpucci}: if $u \le M$ on $\pa\S$ for some constant $M\ge 0$, then $u\le M$  in $\S$. Now we can already prove the desired statements.
	
	\sm
	
	(1) Let $\mu:= \max\limits_{\gad} u = \max\limits_{\gad} g$. If $\mu\ge 0$ then we can apply Theorem~\ref{thpucci} with choice $M:=\mu$, namely, since $u \le \mu$ on $\pa\S$, therefore $u\le \mu$ in $\S$ too, i.e.\
	$$
	\max\limits_{\ov\S} u \le \mu = \max \{ 0, \mu \} = \max \{ 0, \max\limits_{\gad} g \}.
	$$
	If $\mu<0$ then we can apply Theorem~\ref{thpucci} with choice $M:=0$, namely, since $u \le 0$ on $\pa\S$, therefore $u\le 0$ in $\S$ too, i.e.\
	$$
	\max\limits_{\ov\S} u \le 0 = \max \{ 0, \mu \} = \max \{ 0, \max\limits_{\gad} g \}.
	$$
	In both cases we have obtained the desired property (\ref{cmp1}).
	
	\sm
	(2)--(3) These are obvious con\sq s of statement (1).
\end{proof}

\bi
In the special case $q\equiv 0$, the equality (\ref{cmp2}) holds without assuming $g\ge 0$:

\begin{theorem}
	\label{thcmpeq}
	Consider \pr \ (\ref{mixbvpex}) with $q\equiv 0$, under
	the assumptions of Theorem \ref{thcmp}.  That is, let
	(A1)--(A3) hold, $u\in C(\ov\S)$, and assumption (\ref{assqs})  now takes the form
	$f(x) \le 0$ \ ($\forall x\in\S$).
	Then
	\begin{equation}
	\label{cmp2eq}
	\max\limits_{\ov\S} u = \max\limits_{\gad} g .
	\end{equation}
\end{theorem}
\begin{proof}
	If $\max\limits_{\gad} g\ge 0$ then (\ref{cmp1}) implies (\ref{cmp2eq}). Let $\max\limits_{\gad} g< 0$, say, $\max\limits_{\gad} g=-K$ with some $K> 0$. Then the function $w:=u+K$ satisfies the same type of \pr \ with right-hand side  $f$  and b.c. $g+K$, \resp, hence Theorem \ref{thcmp} is valid for this \pr \ as well, and (\ref{cmp1}) for $w$ yields
	$$
	\max\limits_{\ov\S} w \le \max \{ 0, \, \max\limits_{\gad} \, (g+K) \}=0.
	$$
	Then
	$$
	\max\limits_{\ov\S} u \le -K = \max\limits_{\gad} g . \eqno
	$$
\end{proof}

\begin{remark}
	\label{blablaa}
	\rm
	We note that the corresponding minimum principles and nonnegativity property hold
	if the sign conditions on $f$
	are reversed.
\end{remark}

\subsection{Computational scheme: the surface finite element method}
\label{subsscheme}

From now on, we are interested in manifolds that can be embedded in 
$\R^{d+1}$ admitting 
an associated signed
distance function $\delta$, satisfying $\S = \{ x\in \R^{d+1} \mid \delta(x)=0\}$ and $\nb \delta = \nu$. This enables us to use the SFEM setting of \cite{Dziuk88,DziukElliott_acta}.
We note that this embedding assumption is not a very strong restriction in the practically most relevant case $d=2$: as 
described e.g. in \cite{HH},
a wide class of two-dimensional manifolds can be embedded in $\R^{3}$ (globally and isometrically)  under some proper conditions on their smoothness and   Gauss curvature. 

The smooth surface $\S$ with boundary $\pa\S$ is approximated by a triangulated one, denoted by $\S_h$, whose vertices $B_1, B_2, \dotsc, B_{\bar n}$  lie on the surface, including boundary nodes of $\S_h$ lying on $\pa\S$, and it is given as
\begin{equation*}
\S_h = \bigcup_{T \in \Tcalh} T .
\end{equation*}
We always assume that the simplices $T$  form an admissible triangulation (i.e. partition) $\Tcalh$, with $h$ denoting the maximum diameter.  Admissible triangulations were introduced in \cite[Section~5.1]{DziukElliott_ESFEM}: $\S_h$ is a uniform triangulation, i.e.\ every $T \in \Tcalh$ satisfies that the inner radius $\sigma_{h}$ is bounded from below by $ch$ with $c>0$, and $\S_{h}$ is
not a global double covering  of $\S$. The discrete tangential gradient on the discrete surface $\S_h$ is given by
\begin{equation*}
\nb_{\S_h} \phi = \nb {\phi} - (\nb {\phi} \cdot \nu_h)\ \nu_h,
\end{equation*}
understood in a piecewise sense, with $\nu_h$ denoting the normal to $\S_h$, cf.\ \cite{Dziuk88}.
The boundary of the approximation surface is simply $\pa\S_h$, which is an analogous approximation of $\pa\S$.

We define the surface finite element discretisation of our problem, following \cite{Dziuk88} and \cite{boundarySFEM}. We use simplicial elements and continuous piecewise linear functions. The finite element subspace $\vh$ is spanned by the continuous, piecewise linear
basis functions $\chi_j$, satisfying
\begin{equation*}
\chi_j(B_i) = \delta_{ij} \quad  \text{for all }i,j = 1, 2, \dotsc, {\bar n},
\end{equation*}
therefore
\begin{equation*}
\vh = \spn\big\{ \chi_1, \chi_2, \dotsc, \chi_{\bar n} \big\} .
\end{equation*}

Now, let $n < \bar n$ be \sth
\begin{equation*}
B_1, B_2, \dotsc, B_{n}
\end{equation*}
are the vertices that lie on $\S_h$ but not on $\pa\S_h$, and let
\begin{equation*}
B_{{n}+1}, \dotsc, B_{\bar n}
\end{equation*}
be the vertices that lie only on $\pa\S_h$. Then the
basis functions $\chi_1, \chi_2, \dotsc, \chi_{n}$ satisfy the homogeneous
Dirichlet boundary condition on $\pa\S_h$. Hence, we can define
$$
\vhn:={\rm span} \{ \chi_1,\chi_2, \dotsc, \chi_{n} \}\ \subset W^{1,p}_0(S_h).
$$
Further, let
\begin{equation}
\label{ghdef}
\gh = \sum_{j=n+1}^{\bar n} g_j \chi_j \ \in\vh
\end{equation}
(with $g_j\in\er$)  be the piecewise linear \apr ion of the function $g$  on $\pa\S_h$ (and on the
neighbouring elements).

\subsubsection{The lift operator}
\label{subsliftprop}

In the following we recall the so-called  {lift operator}, which was introduced in \cite{Dziuk88} and further investigated in \cite{DziukElliott_ESFEM}. The lift operator projects a finite element function on the discrete surface onto a function on the smooth surface.

Using the oriented distance function $\delta$ (cf.\ \cite[Section~2.1]{DziukElliott_ESFEM}), for a continuous function $\eta_h \colon \S_h \to \R$ its lift is defined as
\begin{equation*}
\eta_{h}^{l}(P) := \eta_h(x), \qquad x\in\S_h,
\end{equation*}
where for every $x\in \S_h$ the value $P=P(x)\in\S$ is uniquely defined via
\begin{equation}
\label{eq: lift defining equation}
x = P + \nu(P) \delta(x) .
\end{equation}
By $\eta^{-l}$ we mean the function whose lift is $\eta$.
Further, we have the lifted finite element spaces
\begin{align*}
V_h^l :=&\ \big\{ \varphi_h = \phi_h^l \, | \ \, \phi_h\in V_h \big\}, \\
(V_h^0)^l :=&\ \big\{ \varphi_h = \phi_h^l \, | \ \, \phi_h\in V_h^0 \big\}.
\end{align*}

The following equivalence between discrete and continuous norms has been shown in \cite{Demlow, Dziuk88}
for functions $\eta_h : \S_h \to \R$   with lift $\eta_h^l : \S \to \R$. {
	This will help to ensure stability of the discrete problem using the surface approximation.
}
\begin{lemma}[equivalence of norms]
	\label{lemma: equivalence of norms}
	The 
	original and lifted norms are equivalent in both the $L^p$ and $W^{1,p}$ case, independently of sufficiently small mesh sizes $h$. That is, there exist  constants $c, h_0>0$ such that for all $0< h\le h_0$,
	\begin{align*}
	c^{-1} \|\eta_h\|_{L^p(\S_h)} \leq &\ \|\eta_h^l\|_{L^p(\S)} \leq c\|\eta_h\|_{L^p(\S_h)} \qquad (\forall \eta_h\in L^p(\S_h)), \\
	c^{-1} \|\eta_h\|_{W^{1,p}(\S_h)} \leq &\ \|\eta_h^l\|_{W^{1,p}(\S)} \leq c\|\eta_h\|_{W^{1,p}(\S_h)} \qquad (\forall \eta_h\in W^{1,p}(\S_h)).
	\end{align*}
\end{lemma}

\subsubsection{The discrete problem}
\label{section: discrete problem}
In order to find the approximate solution of \eqref{mixbvpex}, we have to solve the counterpart of \eqref{wsint}--\eqref{wsbdy} in $\vh$.
We note here that in order to transfer the assumptions on the coefficient functions (A1)--(A4) to the discrete level we have to use negative lifts in the spatial variable. For any $x\in\S_h$ with corresponding lift $P\in\S$, cf.\ \eqref{eq: lift defining equation}, we set
\begin{equation*}
b^{-l}(x,\cdot,\cdot) = b(P,\cdot,\cdot) ,
\end{equation*}
and similarly using the negative lift in the first argument of all coefficient functions ($f$, $q$, $r$, etc.), as well. No lift is used in the other arguments. By this construction all the properties of (A1)--(A4) are clearly transferred to these lifted functions, with $\S_h$ instead of $\S$.

Following \eqref{mixbvpexsubtr}, we first subtract $q^{-l}(x,0)$ from both sides and consider the following discrete problem: find $\uh\in\vh$ such that
\begin{equation}
\label{mixbvph}
\int_{\S_h} \big( b^{-l}(x, \uh, \nbsh \uh)\ \nbsh \uh\cdot \nbsh \vhh + \hat q^{-l}(x,\uh)\vhh \big)
= \int_{\S_h} \hat f^{-l} \vhh   \qquad  (\forall \vhh\in\vhn),
\end{equation}
\begin{equation}
\label{pfth}
\hbox{and} \qquad  u_h-g_h\in \vhn
\end{equation}
where
\begin{equation*}
\hat q^{-l}(x,z):=q^{-l}(x,z)-q^{-l}(x,0), \qquad \hat f^{-l}(x):=f^{-l}(x)-q^{-l}(x,0)
\end{equation*}
for all $x\in\S_h$, $z\in\er$. Further, denote
\begin{equation*}
r^{-l}(x,z):=\left\{
\begin{aligned}
& \frac{q^{-l}(x,z)-q^{-l}(x,0)}{z},  &\qquad& \hbox{if} \ z >  0 , \\
& 0, &\qquad& \hbox{if} \ z \le 0.
\end{aligned}
\right.
\end{equation*}
Then, using also assumption (A4),
\begin{equation}
\label{rnov}
\hat q^{-l}(x,z)=r^{-l}(x,z)\, z\, ,
\qquad
0\le r^{-l}(x,z)\le \alpha   + \beta |z|^{p_1-2} \qquad (\forall x\in{\S_h}, \ z\in\er).
\end{equation}
In what follows, we will rewrite \pr \ (\ref{mixbvph})
as 
\begin{equation}
\label{bvphrew}
\int_{\S_h} \big( b^{-l}(x, \uh, \nbsh \uh)\ \nbsh \uh\cdot \nbsh \vhh +
\, r^{-l}(x,\uh)\uh\vhh \big)
=
\int_{\S_h} \hat f^{-l} \vhh
\qquad  (\forall \vhh\in\vhn) .
\end{equation}
We set
\begin{equation}
\label{uhsum}
\uh =  \sum_{j=1}^{\bar n} c_j \chi_j,
\end{equation}
and look for the \coe s $c_1,c_2,\dotsc,c_{\bar n}$.

\subsubsection{The non\li \ \alb \ \sy}

Now we turn to the non\li \ \alb \ \sy \ \corr \ to \eqref{mixbvph}.
For any ${\bf \bar c} = (c_1,...,c_{\bar n})^T \in \er^{\bar n}$, $i= 1,2,..., n$ and
$j= 1,2,..., \bar n$, we set
\begin{gather*}
b_{ij}({\bf \bar c}) = \int_{\S_h} b^{-l}\Big(x, \sum_{k=1}^{\bar n} c_k \chi_k, \,  \sum_{k=1}^{\bar n} c_k \nbsh \chi_k \Big) \ \nbsh \chi_j \cdot \nbsh \chi_i , \\
r_{ij}({\bf \bar c}) = \int_{\S_h} r^{-l}\Big(x, \sum_{k=1}^{\bar n} c_k \chi_k \Big) \chi_j  \chi_i , \qquad
d_{i} = \int_{\S_h} \hat f^{-l} \chi_i,
\end{gather*}
\begin{equation*}
a_{ij}({\bf \bar c}) = b_{ij}({\bf \bar c}) + r_{ij}({\bf \bar c})  .
\end{equation*}
Putting (\ref{uhsum}) and $\vhh=\chi_i$ into (\ref{bvphrew}), we obtain
the $n\times\bar n$ \sy \ of \alb \ \eq s
\begin{equation}
\label{sy}
\sum_{j=1}^{\bar n} a_{ij}({\bf \bar c}) \, c_j = d_i, \quad i =1,2,...,n.
\end{equation}
Using the notations
\begin{equation*}
\begin{aligned}
{\bf A}({\bf \bar c})=&\ \{a_{ij}({\bf \bar c})\},
\quad \hbox{and} \quad {\bf d} = \{d_j\}, \;  {\bf c} = \{c_{j}\},
& \qquad & i=1,2,...,n, \ j=1,2,...,n, \\
{\bf \tilde A}({\bf \bar c})=&\ \{a_{ij}({\bf \bar c})\},
\quad \hbox{and} \quad {\bf \tilde c} = \{c_{j}\},
& \qquad & i =1,2,...,n, \ j=n+1,...,\bar n ,
\end{aligned}
\end{equation*}
the system (\ref{sy}) turns into
\begin{equation}
\label{syrov}
{\bf A}({\bf \bar c}){\bf c} + {\bf \tilde A}({\bf \bar c}){\bf \tilde c} = {\bf d}.
\end{equation}
Defining further
\begin{equation}
\label{defabar}
{\bf \bar A(\bar c)} = \left[ \begin{matrix} {\bf A(\bar c)} & {\bf \tilde A (\bar c)} \end{matrix}
\right],  \qquad  {\bf \bar c} =  \left[ \begin{matrix} {\bf c} \\ {\bf \tilde c}\end{matrix}
\right],
\end{equation}
we rewrite (\ref{syrov}) as follows
\begin{equation*}
{\bf \bar A(\bar c)}{\bf \bar c} = {\bf d}.
\end{equation*}

In order to obtain a \sy \ with a square  matrix,  we enlarge our \sy \  to an $\bar n\times\bar n$ one. Namely, since $\uh  =   \gh$ on $\pa\S_h$, the coordinates $c_i$ with $ n+1\le i\le \bar n$ satisfy automatically $c_i=g_i$, i.e.\
$$
{\bf \tilde c} = {\bf \tilde g},
$$
where
$$
{\bf \tilde g} = \{g_{j}\},  \quad j=n+1,...,\bar n .
$$
That is, we can replace (\ref{syrov}) by the equivalent \sy
\begin{equation*}
\left[\begin{matrix}
{\bf A}({\bf \bar c}) & {\bf \tilde  A} ({\bf \bar c}) \\
{\bf 0} & {\bf I}
\end{matrix}\right]
\left[\begin{matrix}
{\bf c} \\
{\bf \tilde c}
\end{matrix}\right]
=
\left[\begin{matrix}
{\bf d} \\ {\bf \tilde g}
\end{matrix}\right]  \ .
\end{equation*}

\begin{remark}
	\label{remark: smooth surfaces}
	\rm
	The solution of the arising nonlinear finite element problems can rely on various efficient methods, see e.g. \cite{book}.
\end{remark}

\section{Discrete maximum principles}

\subsection{Classical matrix maximum principles}

\sm
Let us consider a linear algebraic system of equations of order $(n+m) \times (n+m)$:
\begin{equation*}
{\bf \bar A} {\bf \bar c} = {\bf \bar b},
\end{equation*}
where the matrix ${\bf \bar A}$ has the following structure:
\begin{equation}
\label{lin_sys_str}
{\bf \bar A} =
\left[\begin{matrix}
{\bf A} & {\bf \tilde  A} \\ {\bf 0} & {\bf I}
\end{matrix} \right] \ .
\end{equation}
In the above, ${\bf I}$ is an $m \times m$ identity matrix, ${\bf 0}$ is a $m \times n$ zero matrix. In a (surface) finite element setting, such a partitioning arises corresponding to interior and boundary points.

We first recall some classical   definitions and results, see, e.g.\  \cite{Ci70, Varga}. We follow the terminology of \cite{FI10}. Throughout, inequalities for matrices or vectors are understood elementwise, and $\bf c$, $\bf \tilde c$ and $\bf \bar c$  denote    vectors consisting of $n$, $m$ or $n+m$ numbers, \resp.

\begin{definition}
	\label{dwmp}
	\rm The matrix ${\bf \bar A}$ in (\ref{lin_sys_str}) satisfies
	
	\me
	(a) the {\it discrete weak maximum principle (DwMP)} if for an arbitrary vector ${\bf \bar c} = (c_1, ..., c_{n+m})^T \in \er^{n+m}$ satisfying $({\bf \bar A} {\bf \bar c})_i \le 0, \ i =1,2,..., n$, one has
	\begin{equation*}
	\max\limits_{i =1,2,..., n+m} c_i \, \le \, \max \{ 0,
	\max\limits_{i =n+1,..., n+m}  c_i \};
	\end{equation*}
	
	(b) the {\it discrete strict weak maximum principle (DWMP)} if for any vector ${\bf \bar c} = (c_1, ..., c_{n+m})^T \in \er^{n+m}$ satisfying $({\bf \bar A} {\bf \bar c})_i \le 0, \ i =1,2,..., n$, one has
	\begin{equation*}
	\max\limits_{i =1,2,..., n+m} c_i \, = \, \max\limits_{i =n+1,..., n+m}  c_i .
	\end{equation*}
\end{definition}
(DMPs without the term 'weak' also assert that only constant vectors may attain a maximum for 'interior' indices, but we do not address this property here.)

\sm

Based on \cite{CiRav,  FI10} we have the following result, formulated in this way in \cite{KK14}:
\begin{theorem}
	\label{thmmpposdef}
	If the matrix ${\bf \bar A}$ in (\ref{lin_sys_str}) satisfies the following conditions:
	\begin{itemize}
		\item[(i)]  $a_{ij}\le 0$ \qquad ($\forall i=1,\dots,n, \ j=1,\dots, n+m; \ \  i\neq j$),
		\item[(ii)]  $\sum\limits_{j=1  }^{n+m}  a_{ij} \ge 0$ \qquad ($\forall i=1,\dots,n$),
		\item[(iii)] ${\bf A}$ is positive definite,
	\end{itemize}
	then ${\bf \bar A}$  possesses the DwMP.
	
	\me
	If the inequality in condition (ii) is replaced by equality, then ${\bf \bar A}$ possesses the DWMP.
\end{theorem}

\subsection{The discrete maximum principle for the non\li \ \ell \ \pr }

\subsubsection{The main results}

The derivation of our main results needs  the following lemmas on the stability of the boundary data and then on the norm-boundedness of the SFEM solutions. 
This means that we can achieve  stability of the discrete problem, thanks to exploiting  the lift properties of subsection
\ref{subsliftprop} to control the geometric error due to the surface approximation.

\begin{lemma}
	\label{lemma: interp stab}
	Let $g^*$ and $\gamma$ be as given in Assumption \ref{subsubpr} (A2). Then there exists a constant $c > 0$ such that for all $h\leq h_0$ with a $h_0>0$ sufficiently small, the discrete function $g_h$, from \eqref{ghdef}, satisfies
	\begin{equation}
	\label{ghapprox}
	\|g_h\|_{1,p,\S_h} \leq  \ c  \|g^\ast \|_{1,\gamma,\S} \, .
	\end{equation}
\end{lemma}
\begin{proof}
	The assumption on $\gamma$ in (A2)
	implies that in each case we have $\gamma\ge p$ and $\gamma>d$. Since the compact manifold $S_h$ has bounded measure, the relation $\gamma\ge p$ implies that
	$$
	\|g_h\|_{1,p,\S_h} \leq  \ c_1 \|g_h\|_{1,\gamma,\S_h}
	$$
	for some constant $c_1>0$ independent of
	$h$.
	Further, the relation $\gamma>d$ implies  that
	${W^{1,\gamma}(\S_h)}\subset C(\S_h)$, see, e.g., \cite{Ad}. Hence, by \cite[Thm. 3.1.6]{Ci}, the stability of the Lagrange interpolation holds for $(g^\ast)^{-l}$ in $W^{1,\gamma}$-norm on each triangle and hence also on all $\S_h$, therefore
	$$
	\|g_h\|_{1,\gamma,\S_h} \leq  \ c_2  \|(g^\ast)^{-l} \|_{1,\gamma,\S_h}\,
	$$
	for some constant $c_2>0$ independent of $h$.
	Finally,
	Lemma \ref{lemma: equivalence of norms}
	implies that
	$$
	\|(g^\ast)^{-l} \|_{1,\gamma,\S_h} \leq  \ c_3  \|g^\ast \|_{1,\gamma,\S}
	$$
	for some constant $c_3>0$ independent of   $h$. These three estimates together yield the desired result. 
\end{proof}

\begin{lemma}
	\label{lemuhbd}
	The norms 
	$\|\uh \|_{1,p,\S_h}$
	are 
	bounded independently of $h$.
\end{lemma}
\begin{proof}
	We will use the shorthand notation
	$$
	|v_h|_{1,p,\S_h} :=   \|\nbsh v_h\|_{p,\S_h}\,
	$$
	Since \eqref{pfth} shows that $v_h :=u_h-g_h\in \vhn$, we use it as a test function in the discrete weak problem \eqref{bvphrew}. Rearranging the terms, with $g_h$ on the right, we obtain
	\begin{align*}
	&\ \int_{\S_h} \!\!\! \big( b^{-l}(x, \uh, \nbsh \uh)\ {|\nbsh \uh|}^2 + r^{-l}(x, \uh) |\uh|^2 \big) \\
	=&\ \int_{\S_h} \!\!\! b^{-l}(x, \uh, \nbsh \uh) \nbsh \uh  \cdot \nbsh g_h
	+ \int_{\S_h} \!\!\! \Big( r^{-l}(x,\uh) \uh g_h +\, \hat f^{-l} (\uh-g_h)\Big)  .
	\end{align*}
	
	We estimate this equality appropriately from both above and below. For the lower bound of the left-hand side we use (A3), which yields $b^{-l}(x,\cdot,\xi)|\xi|^2\ge \mu_1|\xi|^{p}$ ($x\in\S_h$), and since $r^{-l}(x,\cdot)\ge 0$ ($x\in\S_h$), we obtain the lower bound
	\begin{equation}
	\label{uhpl}
	\mu_1 |u_h|_{1,p,\S_h}^p
	= \mu_1 \int_{\S_h}  |\nbsh u_h|^{p}
	\leq \int_{\S_h}   b^{-l}(x, \uh, \nbsh \uh)\ {|\nbsh \uh|}^2 + r^{-l}(x, \uh) |\uh|^2 .
	\end{equation}
	For the first term on the right-hand side, we have from (A3) that
	\begin{align*}
	\int_{\S_h} \!\!\! b^{-l}(x, \uh, \nbsh \uh) \nbsh \uh  \cdot \nbsh g_h
	\leq &\ \int_{\S_h} \!\!\!  (M_0|\nbsh u_h| + M_1|\nbsh u_h|^{p-1})\, |\nbsh g_h| \\
	\leq &\ M_0 |u_h|_{1,p',\S_h} |g_h|_{1,p,\S_h} + M_1 |u_h|_{1,p,\S_h}^{p-1}  |g_h |_{1,p,\S_h} \\
	\leq &\ c \big(M_0 |u_h|_{1,p,\S_h} + M_1 |u_h|_{1,p,\S_h}^{p-1} \big) \|g_h\|_{1,p,\S_h}
	\end{align*}
	using H\"older's inequality, with $\frac1p + \frac{1}{p'}=1$, $p'\le 2\le p$, and the boundedness of the surface.
	For the second term on the right-hand side, using \eqref{rnov}, we have
	\begin{align*}
	\int_{\S_h} \!\!\!\! \Big( r^{-l}(x,\uh) \uh g_h + \hat f^{-l} (\uh-g_h)\Big)
	= \! &\ \int_{\S_h} \!\!\! \Big( r^{-l}(x,v_h+g_h) (v_h+g_h) g_h +\, \hat f^{-l} v_h\Big) \\
	\leq \! &\ \int_{\S_h} \!\!\! \Big( ( \alpha + \beta |v_h+g_h|^{p_1-2}) (v_h+g_h) g_h + \hat f^{-l} v_h\Big) \\
	\leq \! &\ \big(\hat k_1 |v_h|_{1,p,\S_h}+ \|g_h\|_{1,p,\S_h} \big) \|g_h\|_{1,p,\S_h} \\
	&\ + \big( \hat k_1 |v_h|_{1,p,\S_h}   + \|g_h\|_{1,p,\S_h}\big)^{p_1-1} \|g_h\|_{1,p,\S_h} \\
	&\ + c\hat k_1 \|\hat f\|_{p',\S} |v_h|_{1,p,\S_h} ,
	\end{align*}
	where the last estimates is obtained, similarly as before, using repeatedly H\"older's inequality (with $p'\le 2\le p$) and again the boundedness of the surface, and the equivalent version of \eqref{Sobombas} on $W_0^{1,p}(\S_h)$:
	\begin{equation*}
	\|v\|_{p_1,\S_h}  \le \hat k_1 |v |_{1,p,\S_h},
	\qquad (\forall v\in W_0^{1,p}(\S_h)),
	\end{equation*}
	together with the bound $\|\hat f^{-l}\|_{p',\S_h} \leq c \|\hat f\|_{p',\S}$.
	This can be further estimated using
	$$
	|v_h|_{1,p,\S_h} \leq |u_h|_{1,p,\S_h} + |g_h |_{1,p,\S_h}  \leq |u_h|_{1,p,\S_h} + \|g_h\|_{1,p,\S_h}
	$$
	in each term. Altogether, we obtain
	\begin{align*}
	|u_h|_{1,p,\S_h}^p \leq &\ c\, \Big( \big( |u_h|_{1,p,\S_h} + 2\|g_h\|_{1,p,\S_h} \big) \|g_h\|_{1,p,\S_h} \\
	&\ + \big(|u_h|_{1,p,\S_h} + \|g_h\|_{1,p,\S_h} + \|g_h\|_{1,p,\S_h}\big)^{p_1-1} \|g_h\|_{1,p,\S_h} \\
	&\ + c \|\hat f\|_{p',\S} (|u_h|_{1,p,\S_h} + \|g_h\|_{1,p,\S_h}) \Big),
	\end{align*}
	a polynomial of degree $p-1$ of $|u_h|_{1,p,\S_h}$ on the right-hand side, whose \coe s, containing $\|\hat f\|_{p',\S}$ and $\|g_h\|_{1,p,\S_h}$, are bounded independently of $h$, owing to \eqref{ghapprox}. Further, we have a power $p$ of $|u_h|_{1,p,\S_h}$ on the left-hand side. Since the latter is bounded by the former, this implies that $|u_h|_{1,p,\S_h}$ is bounded. Then $|v_h|_{1,p,\S_h} \leq |u_h|_{1,p,\S_h} + |g_h|_{1,p,\S_h}$ is also bounded (using again \eqref{ghapprox} for the term with $g_h$). Here $v_h\in \vhn$, where the norms $\|\cdot \|_{1,p,\S_h}$ and $|\cdot|_{1,p,\S_h}$ are equivalent, hence $\|v_h\|_{1,p,\S_h}$  is also bounded by some $K>0$. Finally, due to Lemma~\ref{lemma: interp stab},
	$$
	\|u_h\|_{1,p,\S_h} \le \|v_h\|_{1,p,\S_h} + \|g_h\|_{1,p,\S_h}  \leq K + c \|g^*\|_{1,\gamma,\S}\, ,
	$$
	i.e. it is  also bounded independently of $h$.
\end{proof}

Now we can verify the underlying matrix maximum principle for the stiffness matrix of our \ell \ \pr .
\begin{theorem}
	\label{thdmp1}
	Let (A1)--(A4) hold and let
	us consider a family of simplicial triangulations ${\cal T}_h$ ($h>0$) satisfying the following property: for any $i =1,2,..., n, \ j =1,2,...,\bar n  \ (i\neq j)$
	\begin{equation}
	\label{nbacute}
	\nbsh \chi_i \cdot \nbsh \chi_j \le -\frac{\sigma_0}{h^2}<0
	\end{equation}
	on $\supp \chi_i \ \cap \ \supp \chi_j \subset \S_h$ with $\sigma_0>0$ independent of  $i,j$     and $h$.
	Let the simplicial triangulations ${\cal T}_h$ be regular, i.e.\ there exist constants $m_1, m_2>0$ \sth \ for any $h>0$ and any simplex $T_h\in {\cal T}_h$
	\begin{equation}
	\label{meshreg}
	m_1 h^d \le \meas(T_h) \le m_2 h^d \,
	\end{equation}
	(where $\meas(T_h)$ denotes the $d$-\dim al measure of $T_h$).
	
	\bi
	
	Then for sufficiently small $h$, the matrix ${\bf \bar A(\bar c)}$
	defined in (\ref{defabar}) has the following properties:
	\begin{enumerate}
		\item[(i)] \ $a_{ij}({\bf \bar c}) \le 0, \quad i =1,2,\dotsc, n, \ j =1,2,\dotsc,\bar n$ such that $i\neq j$.
		\item[(ii)] \ $\sum\limits_{j=1}^{\bar n} a_{ij}({\bf \bar c})\ge 0, \quad i =1,2,\dotsc,n$.
		\item[(iii)] \ ${\bf A}(\bar c)$ is positive definite.
	\end{enumerate}
\end{theorem}

\smallskip

\begin{proof}
	Let us recall that for any $i,=1,2,\dotsc, n$, $j=1,2,\dotsc,\bar n$,
	$$
	a_{ij}({\bf \bar c}) =   \int_{\S_h} \Big( b^{-l}(x, \uh,  \nbsh \uh) \nbsh \chi_i \cdot \nbsh \chi_j \,
	+ r^{-l}(x,\uh) \chi_i  \chi_j \Big) .
	$$
	Now we   prove  properties (i)--(iii).
	
	\bi
	(i) Let $i =1,2,\dotsc, n$, $j =1,2,\dotsc,\bar n$ with $i\neq j$ and let $\S_h^{ij}$ denote the interior of $\supp \chi_i \cap \supp \chi_j \subset \S_h$.
	If $\S_h^{ij}=\emptyset$ then
	$$
	a_{ij}({\bf \bar c})=0.
	$$
	If $\S_h^{ij}\neq\emptyset$ then we use property (A3), namely that $b(\cdot,u_h,\nbsh u_h) \geq \mu_0$, and \eqref{nbacute}, further,  $r(\cdot, \uh) \geq 0$ and the fact $0 \leq \chi_i \leq 1$ ($i=1,2,\dotsc,\bar n$), which imply
	\begin{equation}
	\label{sigb1}
	a_{ij}({\bf \bar c}) \le  -\, \frac{\sigma_0}{h^2} \, \mu_0 \,
	\meas(\S_h^{ij}) + \int_{\S_h^{ij}} r^{-l}(x,\uh)
	.
	\end{equation}
	Here, from \eqref{rnov},
	$$
	\int_{\S_h^{ij}}r^{-l}(x,\uh) \leq \int_{\S_h^{ij}} \Bigl( \alpha + \beta |\uh|^{p_1-2}\Bigr) .
	$$
	We now further estimate the integral from above. From assumption (A4) we have $2\le p_1<\hat p$, where $\hatp=\frac{2\dimn}{\dimn-p}$ (if $d > p$)  and $\hat p := +\infty$ (if $d \leq p$).
	Let us also consider the Sobolev embedding estimate \eqref{Sobombas}, where
	$p^\ast:=\frac{p\dimn}{\dimn-p}$  (if $\dimn>p$) and $p^\ast:=\infty$ (if $\dimn\le p$) were defined.
	Since $2 \leq p$ from (A1),  we clearly have
	\begin{equation}
	\label{eq: new index relations}
	\hatp \leq \pcs . 
	\end{equation}
	Hence $p_1$ also satisfies $p_1<\pcs$, which implies that the Sobolev embedding estimate \eqref{Sobombas} holds with this $p_1$. Now, assume first that $p_1 > 2$ and let us fix a real number $r$ satisfying
	\begin{equation}
	\label{rtprop}
	\frac{d}{2} < r \leq \frac{\pcs}{p_1-2} \ .
	\end{equation}
	Such a number $r$ always exists, since
	for $d\le p$ there is no upper bound ($p^\ast:=\infty$), and for $d>p$
	the assumption $p_1<\hat p$ implies
	$$
	p_1-2 < \hat p-2 = \frac{2p }{\dimn-p} = \frac{2\pcs }{\dimn}
	$$
	which yields
	$$
	\frac{d}{2}  <
	\frac{\pcs}{p_1-2} .
	$$
	Further, the assumption $d \geq 2$ on the dimension of the manifold implies $r > 1$.
	(In theory the trivial case $d=1$ is also allowed,   then we must assume that $r > 1$ is chosen, which is allowed since $\frac{\pcs}{p_1-2}>1$.)

	Now let $s>1$ be chosen such that $\frac{1}{r}+\frac{1}{s}=1$. Then H\"older's inequality implies
	\begin{equation}
	\label{holdimp}
	\int_{\S_h^{ij}}|\uh|^{p_1-2} \le
	\|1\|_{L^{s}(\S_h^{ij})} \, \Bigl\||\uh|^{p_1-2}\Bigr\|_{L^{r}(\S_h^{ij})} = \meas(\S_h^{ij})^{1/s} \, \|\uh \|^{p_1-2}_{L^{(p_1-2)r}(\S_h^{ij})}\, .
	\end{equation}
	Here $(p_1-2)r\le\pcs$ and \eqref{Sobombas} imply
	$$
	\|\uh \|^{p_1-2}_{L^{(p_1-2)r}(\S_h^{ij})} \leq \|\uh \|^{p_1-2}_{L^{(p_1-2)r}(\S)}
	\leq k_1^{p_1-2} \|\uh \|_{1,p,\S}^{p_1-2}\, .
	$$
	Here, by Lemma \ref{lemuhbd}, $\|\uh \|_{1,p,\S}$ is bounded independently of $h$.
	Hence,
	(\ref{holdimp}) is bounded as
	\begin{equation}
	\label{holdimpom}
	\int_{\S_h^{ij}}|\uh|^{p_1-2} \le K_1 \meas(\S_h^{ij})^{1/s}
	\end{equation}
	with some constant $K_1>0$ independent of $h$.
	Finally, if $p_1=2$ then the \corr \ equality (\ref{holdimpom})
	holds trivially with $s=1$.
	
	The integral  of the positive constant $\alpha$ 
	is simply
	$$
	\int_{\S_h^{ij}} \alpha = \alpha \meas(\S_h^{ij}) .
	$$
	Substituting the two estimates into \eqref{sigb1}, we obtain
	\begin{equation*}
	a_{ij}({\bf \bar c}) \leq \bigg( - \frac{\sigma_0\mu_0}{h^2} + \alpha \bigg)\meas(\S_h^{ij}) + \beta K_1 \meas(\S_h^{ij})^{1/s}  .
	\end{equation*}
	
	For  sufficiently small $h$, we can write
	$$
	a_{ij}({\bf \bar c}) \leq A^{ij}(h) := - \frac{C_0}{h^2} \meas(\S_h^{ij}) + C_1 \meas(\S_h^{ij})^{1/s} ,
	$$
	more precisely, there exist positive constants $h_0$, $C_0$ and $C_1$  independently of $h$ and $i,j$  such that the above inequality holds for $h<h_0$.
	Then, using that $\frac{1}{r}+\frac{1}{s}=1$ and the regularity \eqref{meshreg} of the mesh,
	we have
	\begin{align*}
	A^{ij}(h) =&\ \meas(\S_h^{ij})^{1/s} \Big( - \frac{C_0}{h^2} \meas(\S_h^{ij})^{1/r} +  C_1 \Big) \\
	\leq &\ \meas(\S_h^{ij})^{1/s} \Big( - C_2  h^{-2+(d/r)}  +  C_1 \Big) .
	\end{align*}
	Since \eqref{rtprop} implies $\frac{d}{r}<2$, the term in brackets tends to $-\infty$ as $h \to 0$ and hence $A^{ij}(h)<0$ for $h<h_0$. 
	
	\sm
	
	Altogether, we obtain  that there exists $h_0>0$ \sth\ for $h \leq h_0$ and all $i,j$
	\begin{equation*}
	a_{ij}({\bf \bar c}) < 0 .
	\end{equation*}

	(ii) \ For any $i =1,2,\dotsc,n$,
	\begin{equation}
	\label{asum}
	\begin{aligned}
	\sum\limits_{j=1}^{\bar n} a_{ij}({\bf \bar c}) =&\ \int_{\S_h} \Big( b^{-l}(x, \uh,  \nbsh \uh) \ \nbsh \chi_i \cdot \nbsh \big(\sum\limits_{j=1}^{\bar n}\chi_j\big) + r^{-l}(x,\uh) \   \chi_i  \big(\sum\limits_{j=1}^{\bar n} \chi_j \big) \Big) \\ 
	=   &\ \int\limits_{\S_h} r^{-l}(x,\uh) \,   \chi_i
	\geq 0,
	\end{aligned}
	\end{equation}
	using the fact that $\sum_{j=1}^{\bar n}\chi_j\equiv 1$ and that $r, 
	\chi_i$
	are nonnegative.
	
	\sm
	
	(iii) For any vector $\bd\in \er^n, \ \bd\neq 0$ and \corr \ finite element function $v_h=\sum_{j=1}^{  n} d_j \chi_j\neq 0$, using that $b\ge \mu_0$ and $r\ge 0$, we have
	\begin{align*}
	{\bf  A({\bf \bar c})}\bd\cdot \bd
	= \sum_{i,j=1}^{  n} a_{ij}({\bf \bar c}) d_j d_i
	=&\ \int_{\S_h} \Bigl(  b^{-l}(x, \uh,  \nbsh \uh)\ |\nbsh \vhh|^2 + r^{-l}(x,u_h)u_h^2\Bigr) \\
	\geq &\ \mu_0 \int_{\S_h}   |\nbsh \vhh|^2 >0
	\end{align*}
	hence ${\bf A}(\bar c)$ is positive definite.
\end{proof}

\bi

Now it is straightforward to derive the analogue of Theorem  \ref{thcmp} 
for   \sy \
(\ref{syrov}), i.e.\ the discrete  maximum principle.

\begin{theorem}[Discrete  maximum principle]
	\label{thdmp2}
	Let  \pr \ (\ref{mixbvpex}) satisfy
	\begin{equation}
	\label{assqsmax}
	f(x)-q(x,0)\le 0 \quad (x\in\S)
	\end{equation}
	and consider the discretization given in subsection
	\ref{subsscheme}.    Under the conditions of Theorem \ref{thdmp1}, we have
	\begin{equation}
	\label{dmph1}
	\max\limits_{\ov\S} \uh^l \le   \max \{ 0, \max\limits_{\gad} g_h^l \} .
	\end{equation}
	In particular, if   $g\ge 0$  then
	\begin{equation}
	\label{dmph2}
	\max\limits_{\ov\S}\uh^l  = \max\limits_{\gad} \gh^l  ,
	\end{equation}
	and  if   $g\le 0$ then
	we have the discrete nonpositivity property
	\begin{equation}
	\label{dnpp}
	\max\limits_{\ov\S} \uh^l \le 0.
	\end{equation}
\end{theorem}

\begin{proof}
	It follows 
	from
	Theorem \ref{thdmp1} and the properties of piecewise linear functions. Namely, Theorem \ref{thdmp1} states that the matrix ${\bf \bar A}({\bf\bar c})$ satisfies the conditions of Theorem \ref{thmmpposdef}, further, \eqref{assqsmax} implies that $ ( {\bf \bar A}({\bf\bar c}){\bf\bar c})_i = d_{i} = \int_{\S_h} \hat f^{-l}\chi_i\le 0$. Hence, for our system ${\bf \bar A(\bar c)}{\bf \bar c} = {\bf d}$, Theorem \ref{thmmpposdef} 
	provides the algebraic matrix maximum principle  DwMP,  
	which is equivalent to
	\begin{equation}
	\label{dmph1shnode}
	\max\limits_{B_i\in \ov\S_h} \uh(B_i)  \le   \max \{ 0, \max\limits_{B_i\in {\gad}_h} g_h(B_i) \}
	\end{equation}
	for the node points. Since $u_h$ is a piecewise linear function, hence \eqref{dmph1shnode} 
	is also true in the elements  between the node points, i.e. on the whole domain $\ov\S_h$:
	\begin{equation}
	\label{dmph1sh}
	\max\limits_{\ov\S_h} \uh  \le   \max \{ 0, \max\limits_{{\gad}_h} g_h \}
	\end{equation}
	The lift preserves extreme values, hence  \eqref{dmph1sh} implies \eqref{dmph1}.
	Finally, \eqref{dmph2} and \eqref{dnpp} are trivial con\sq s of \eqref{dmph1}.
\end{proof}

\subsubsection{Related results}

One can verify in the same way the discrete minimum principle for \sy \ (\ref{syrov}), in analogy with Theorem \ref{thdmp2}:

\begin{theorem}
	\label{thdminp2}
	Let  \pr \ \eqref{mixbvpex} satisfy
	\begin{equation*}
	f(x)-q(x,0)\ge 0 \quad (x\in\S)
	\end{equation*}
	and consider the discretisation given in subsection \ref{subsscheme}. Under the conditions of Theorem \ref{thdmp1}, we have
	\begin{equation*}
	\min\limits_{\ov\S} \uh^l \ge   \min \{ 0, \min\limits_{\gad} g_h^l \} .
	\end{equation*}
	In particular, if   $g\le 0$  then
	\begin{equation*}
	\min\limits_{\ov\S}\uh^l  = \min\limits_{\gad} \gh^l  ,
	\end{equation*}
	and  if   $g\ge 0$ then
	we have the discrete nonnegativity property
	\begin{equation*}
	\min\limits_{\ov\S} \uh^l \ge 0.
	\end{equation*}
\end{theorem}

In the special case $q\equiv 0$, the counterpart of Theorem
\ref{thcmpeq} is valid, i.e.\ equality   holds without assuming any sign condition on $g$. We formulate this for both the maximum and minimum principles.
Moreover,  the strict negativity in (\ref{nbacute}) can be
replaced by a weaker nonnegativity condition,
and no regularity condition on the mesh like (\ref{meshreg}) 
needs to be assumed.

\begin{theorem}
	\label{thdmpeq}
	Let us consider the following special case  of \pr \ (\ref{mixbvpex}):
	\begin{equation}
	\label{mixbvpspec}
	\left\{
	\begin{aligned}
	-\div_{\S} \Bigl(b(x,u,\nbs u)\, \nbs u\Bigr)   =&\ f(x) &\qquad& {\rm on } \, \S,
	\\
	u   =&\  g(x)   &\qquad &  {\rm on} \
	\partial\S .
	\end{aligned}
	\right.
	\end{equation}
	under assumptions (A1)--(A3).
	Let
	the triangulation ${\cal T}_h$  satisfy the following property:
	for any $i =1,2,..., n, \ j =1,2,...,\bar n  \ (i\neq j)$
	\begin{equation}
	\label{nbacutew}
	\nbsh \chi_i \cdot \nbsh \chi_j \le  0 .
	\end{equation}
	Then the following results hold:
	
	\me
	
	(1)  If $f\le 0$ 
	then \, $\disp \max\limits_{\ov\S}\uh^l  = \max\limits_{\gad}
	\gh^l$.
	
	(2) If $f\ge 0$ 
	then \, $\disp \min\limits_{\ov\S}\uh^l  = \min\limits_{\gad}
	\gh^l$.
	
	(3) If $f= 0$ 
	then the ranges of $u_h^l$ and $g_h^l$
	coincide, i.e.\ we have $[\min\limits_{\ov\S}\uh^l, \max\limits_{\ov\S}\uh^l ]=[\min\limits_{\gad} \gh^l, \max\limits_{\gad} \gh^l]$
	for the \corr \ intervals.
\end{theorem}

\begin{proof}
	To verify (1), we rely on Theorem \ref{thmmpposdef} again, whose
	conditions  follow similarly as in Theorem \ref{thdmp1}.
	The difference in the proof arises in proving property (i), i.e.\
	\begin{equation}
	\label{aijneg2}
	a_{ij}({\bf \bar c})\le 0,
	\end{equation}
	which now follows trivially: since the assumption  $q\equiv 0$  implies $r\equiv 0$, we simply have
	$$
	a_{ij}({\bf \bar c}) =   \int_{\S_h}   b^{-l}(x, \uh,  \nbsh \uh) \nbsh \chi_i \cdot \nbsh \chi_j
	$$
	and thus (\ref{nbacutew}) and the condition $b\ge \mu_0>0$ readily yield \eqref{aijneg2}.
	Further, in property (ii), (\ref{asum}) is simply replaced by
	$$
	\sum\limits_{j=1}^{\bar n} a_{ij}({\bf \bar c}) = \ \int_{\S_h}   b^{-l}(x, \uh,  \nbsh \uh) \ \nbsh \chi_i \cdot \nbsh \big(\sum\limits_{j=1}^{\bar n}\chi_j\big)  = 0,
	$$
	hence we can apply the final statement  of Theorem \ref{thmmpposdef} to obtain that DWMP holds.
	This imnplies, similarly to the argument of Theorem
	\eqref{thdminp2} on piecewise linear functions, that
	$\disp \max\limits_{\ov\S}\uh^l  = \max\limits_{\gad}
	\gh^l$.
	
	\sm
	
	Statement (2) follows from (1) by replacing $u$ by $-u$, and (3)
	is a direct con\sq \ of (1) and (2).
\end{proof}

\subsubsection{On geometric mesh properties}
\label{section:on acute meshes}

Conditions (\ref{nbacute}) and (\ref{nbacutew})
are easy to check, since the values $\nabla_{S_h} \chi_i \cdot \nabla_{S_h} \chi_j$ are constant  on each element.
These conditions have a straightforward geometric interpretation:
the angles must be uniformly acute for (\ref{nbacute}) and nonobtuse for (\ref{nbacutew}).

The currently available constructions of acute triangulations on surfaces of various types
and close issues have been recently surveyed in the paper \cite[Section 3]{Zamf}, see also
the literature therein. We may notice as well that the acuteness property is stable with respect
to small perturbations of the vertices involved, therefore e.g. constructions of acute
triangulations on planes (see \cite{Zamf, BraKorKriSol}) can be easily adapted for producing
acute triangulations for  certain (not very curved) surfaces in 3D.

We underline that many real-life problems appearing in biology, biophysics and biochemistry
are posed on sphere-like surfaces. In typical cases the surface at hand there is a membrane,
a surface of a cell, or of a nucleous, a surface of a crystal, or some other spherical object,
see for instance \cite{DziukElliott_ESFEM,DziukElliott_acta} and the references therein.
Acute triangulations and refinement on a spherical surface will be discussed in Section \ref{secnumexp}.

For general curved surfaces, the issue of refinements
preserving  acuteness (and the construction of families of acute triangulations)
is a more involved task than the case of nonobtuse triangulations (similarly
to the case of Euclidean domains). For the latter, for example,
a constructive proof of existence of a family of nonobtuse triangulations of surfaces of cylindric-type
3D domains is given in \cite{KorAM2012}. It is based on the construction of conforming nonobtuse tetrahedral
meshes and the fact that faces of nonobtuse tetrahedra are nonobtuse triangles.

A way of general remedy in such mesh generation is improving mesh properties using arbitrary Lagrangian Eulerian maps. Such an approach has been developed in \cite{ALEmap} to generate meshes with acute angles for general  closed  surfaces. The algorithm there is based on a constraint system and on arbitrary Lagrangian Eulerian (ALE) maps, see in particular Section~5.3 therein for meshes with angle conditions. This technique can be adapted   to   surfaces with boundary, see Section 
\ref{secnumexp} below where we have done this for our numerical tests.

\subsection{Examples}
\label{section: examples}

We may mention some important real-life examples from the paper \cite{Pucci} where the \corr \ CMP has been proved. The arising equations include the following models:

\begin{enumerate}
	\item[(i)]   gas dynamics:
	$$
	-\div_{\S} \Bigl( \varrho(|\nb_\S u|^2)\, \nbs u\Bigr)   = \, 0
	$$
	where the function $\varrho$, which describes the relation of the velocity  and the density,  is determined by  Bernoulli's law;
	
	\item[(ii)]  surface $p$-Laplacian:
	$$
	-\div_{\S} \Bigl(  |\nb_\S u|^{p-2}\, \nbs u\Bigr)   = \, 0
	$$
	which minimizes the $p$-Dirichlet norm on $S$;
	
	\item[(iii)]   radiative cooling:
	$$
	-\div_{\S} \Bigl( \kappa(x,u)|\nb_\S u|^{p-2}\, \nbs u\Bigr) + \sigma u^4  = \, 0
	$$
	where $\kappa$ is the coefficient of heat conduction and $\sigma$  is the radiation, assumed to be constant. (Here $u\ge 0$ is of physical interest \cite{Keller}, hence the non\li ity is defined as $q(x,z):=z^4$ for $z\ge 0$ only and as $q(x,z)\equiv 0$ for $z\le 0$.)
\end{enumerate}

According to our results, if the finite element discretization of the \corr \ \bvp \ satisfies
the angle conditions   described in subsection \ref{section:on acute meshes}, then the numerical solution satisfies the DMP. The typical situation is that the boundary function $g$ describes a nonnegative physical quantity: $g\ge 0$.
Then Theorem
\ref{thdminp2} ensures the discrete nonnegativity property, i.e. that the numerical solution satisfies \eqref{dnpp}:
\begin{equation*}
\min\limits_{\ov\S} \uh^l \ge 0
\end{equation*}
in accordance with the physical reality.

\section{Numerical experiments}
\label{secnumexp}

In this section we present illustrative results of numerical tests performed for radiative cooling and $p$-Laplacian models from Section~\ref{section: examples}  posed on different surfaces with boundary.  As a conclusion, we may observe the validity of the discrete maximum-minimum principle in each numerical example.

\subsection{Generating acute meshes}

As  described  in Subsection~\ref{section:on acute meshes}, the generation of acute  or nonobtuse  surface triangulations need special care. We first show examples where the particular surface properties can be exploited to generate such meshes directly, and then we also demonstrate how to apply the recently developed   generation approach \cite{ALEmap}. The studied surfaces are thus a hemisphere, a semi-torus, and an elaborate surface with four holes.

For the case of the hemisphere, the surface was approximated by subsequent refinements of the initial grid, taken to be the half of an icosahedron. In the refinement step each triangle was first divided into four similar triangles and then newly generated vertices were projected back to the surface.

For the case semi-torus, the meshes were generated via an initial subdivision   into round strips, such that  the vertices on one side of each strip lie on equal distances and on the other side similarly but in "chess-order", and finally linking the vertices as in Figure \ref{figure:plots} (middle).

For the elaborate surface with four holes
(a variant with boundary of an example in \cite{ElliottVenkataraman_ALE}),
we adapted the techniques of  \cite{ALEmap} to  surfaces with boundary, also handling the extra difficulty of pairs of triangles with a joint edge (almost) perpendicular to the boundary, which are not accounted by the original algorithm.  
Using this method we generated triangular meshes with acute angles.  The surface with four holes is associated with the distance function  
\begin{align*}
\delta(x) = G(x_1^2) + G(x_2^2) + \frac{x_3^2}{0.1^2} - 1 ,
\end{align*}
with $G(s) =  31.25 s (s - 0.36) (s - 0.95)$, then the surface is defined by
\begin{align*}
S = \{ x \in \R^3 \mid  \delta(x) = 0 , \ x_3 \geq 0 \} .
\end{align*}

The quality of all three meshes, in terms of minimal and maximal angles, is shown in Figure~\ref{figure: mesh quality}.

\begin{figure}[ht!]
	\centering
	\includegraphics[width=\textwidth,height=0.3\textheight]{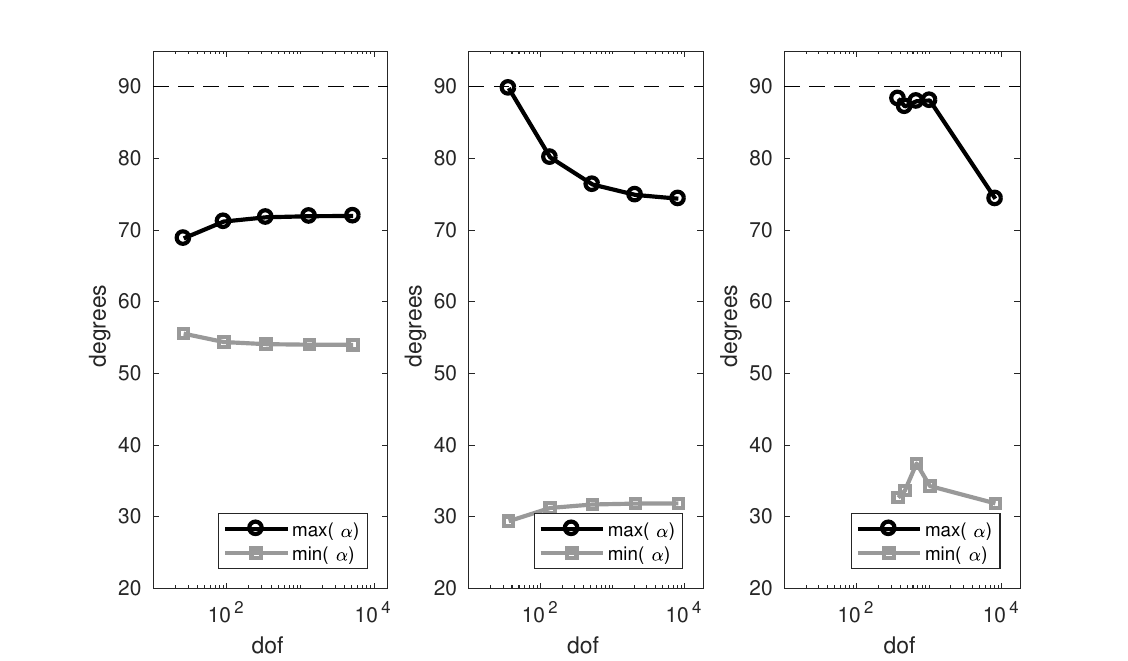} \\
	\caption{The minimum and maximum angles of the meshes generated for the test problems posed on the hemisphere (left), the semi-torus (middle) and the surface with four holes (right).}
	\label{figure: mesh quality}
\end{figure}

\subsection{Radiative cooling on a hemisphere}

First, we carry out some numerical experiments for the radiative cooling model from (iii) of  subsection~\ref{section: examples}
with $\kappa=1$ and $p=2$:
\begin{equation}
\label{eq:RC problem}
\begin{aligned}
-\Delta_{\S} u + \sigma u^4  = &\  0 \qquad \textrm{ in } \ \S, \\
u =&\ g  \qquad \textrm{ on } \ \partial\S ,
\end{aligned}
\end{equation}
with the nonlinearity defined as $q(x,z):=z^4$ for $z\ge 0$ and $q(x,z)\equiv 0$ for $z\le 0$. We choose $\sigma=5$,
and select the boundary function as
\begin{equation*}
g(x,y,z)= 1 +  xy.
\end{equation*}
We assume that $\S$ is a hemisphere of the radius $1$, therefore $\textnormal{Ran}(g) = [0.5,1.5]$, in particular, $g$
is a nonnegative function over the boundary.
This problem, as well as the ones below, were solved with a damped Newton iteration.

The minimum and maximum values of the numerical solutions for the radiative cooling problem on different meshes are presented in Table~\ref{table:dmp RC} (along the minimum and maximum angles in the mesh).
\begin{table}[htbp]
	\centering
	\begin{tabular}{ r c l l l l }
		\toprule
		dof & & $\min\{\alpha\}$ $\qquad$ & $\max\{\alpha\}$ $\qquad$ & $\min\{u_h\}$ $\qquad$ & $\max\{u_h\}$ \\
		\midrule
		91 & & 54.39 & 71.20 & 0.5041 & 1.4755\\
		341 & & 54.09 & 71.80 & 0.5 & 1.5\\
		1321 & & 54.02 & 71.95 & 0.5 & 1.5\\
		5201 & & 54.00 & 71.98 & 0.5 & 1.5\\
		\bottomrule
	\end{tabular}
	\caption{Minimum and maximum values of the numerical solutions for the radiative cooling problem on the hemi-sphere}
	\label{table:dmp RC}
\end{table}

%


\subsection{$p$-Laplacian on a semi-torus}

On the semi-torus we performed the experiment for the $p$-Laplacian  with $p=4$, i.e.
\begin{align*}
-\div_{\S} \big(  |\nb_\S u|^{2} \ \nbs u\big)  = &\  0 \qquad \textrm{ in } \ \S, \\
u =&\ g  \qquad \textrm{ on } \ \partial\S .
\end{align*}
Homogeneous $p$-Laplacian equations with
$p=4$ arise, e.g.~in rheology, see \cite{BusCio}.
We prescribe the following boundary data:
\begin{equation*}
g(x,y,z)= 10 + x, \qquad \textnormal{ hence }  \quad \textnormal{Ran}(g) = [3,17] .
\end{equation*}

The discrete maximum principle for the $p$-Laplace problem is illustrated by Table~\ref{table:dmp pLaplace}, similarly as in the previous table before.
\begin{table}[htbp]
	\centering
	\begin{tabular}{ r c l l l l }
		\toprule
		dof & & $\min\{\alpha\}$ $\qquad$ & $\max\{\alpha\}$ $\qquad$ & $\min\{u_h\}$ $\qquad$ & $\max\{u_h\}$ \\
		\midrule
		36 & & 29.39 & 89.82 & 3.13397 & 17 \\
		136 & & 31.24 & 80.16 & 3.00629 & 17 \\
		528 & & 31.72 & 76.37 & 3.00034 & 17 \\
		2080 & & 31.85 & 74.92 & 3.00002 & 17 \\
		8256 & & 31.88 & 74.39 & 3.00000 & 17 \\
		\bottomrule
	\end{tabular}
	\caption{Minimum and maximum values of the numerical solutions for the $p$-Laplace problem on the semi-torus}
	\label{table:dmp pLaplace}
\end{table}

%

\subsection{Radiative cooling  on a surface with four holes}

Let us consider again the radiative cooling problem \eqref{eq:RC problem} on the surface
\begin{equation}
\label{eq:surface 4H}
S = \{ x \in \R^3 \mid \delta(x) = 0 , \ x_3 \geq 0 \} ,
\end{equation}
with the distance function $\delta(x) = G(x_1^2) + G(x_2^2) + \frac{x_3^2}{0.1^2} - 1$, and where $G(s) =  31.25 s (s - 0.36) (s - 0.95)$.
We prescribe the following boundary data on $\pa  S$:
\begin{equation*}
g(x,y,z)= 1 + xy, \qquad \textnormal{ hence }  \quad \textnormal{Ran}(g) = [0.02345468,1.97653755] .
\end{equation*}
Here the minimum and maximum values over the boundary were obtained by numerically computing the extrema of the constrained problem.

Similarly to the previous tables before, the minimum and maximum values of the numerical solutions for the radiative cooling problem on different meshes approximating the surface with four holes 
are presented in Table~\ref{table:dmp RC 4H_half}.
\begin{table}[htbp]
	\centering
	\begin{tabular}{ r c l l l l }
		\toprule
		dof & & $\min\{\alpha\}$ $\qquad$ & $\max\{\alpha\}$ $\qquad$ & $\min\{u_h\}$ $\qquad$ & $\max\{u_h\}$ \\
		\midrule
		375 & & 32.68 & 88.34 & 0.05137 & 1.94184\\
		464 & & 33.66 & 87.26 & 0.04997 & 1.94978\\
		672 & & 37.45 & 87.99 & 0.02776 & 1.97317\\
		1025 & & 34.30 & 88.10 & 0.02886 & 1.97553\\
		\bottomrule
	\end{tabular}
	\caption{Minimum and maximum values of the numerical solutions for the radiative cooling problem on the surface with four holes 
	}
	\label{table:dmp RC 4H_half}
\end{table}

\bi 

Finally, Figure~\ref{figure:plots} reports on the surface meshes and on the numerical solutions for all three examples. The plotted results correspond to the last rows in each table, respectively.

\begin{figure}[htbp]
	\centering
	\includegraphics[width=.7\textwidth,height=0.26\textheight]{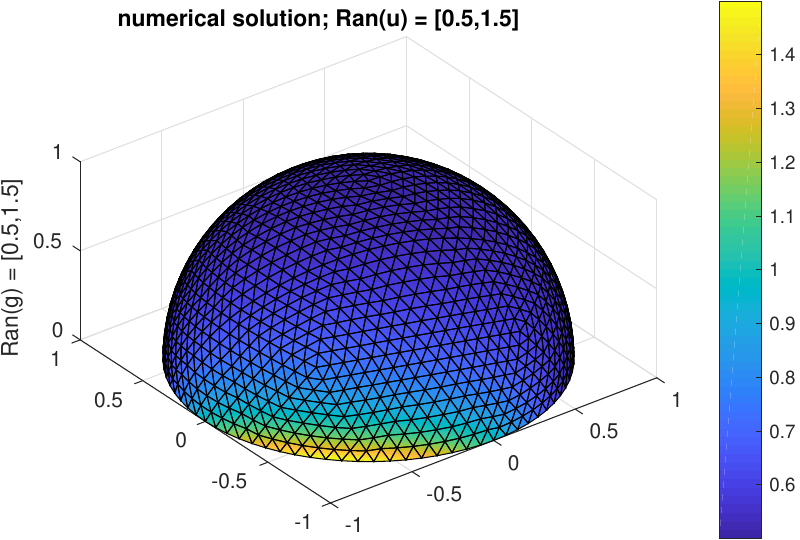} \\
	\includegraphics[width=.7\textwidth,height=0.26\textheight]{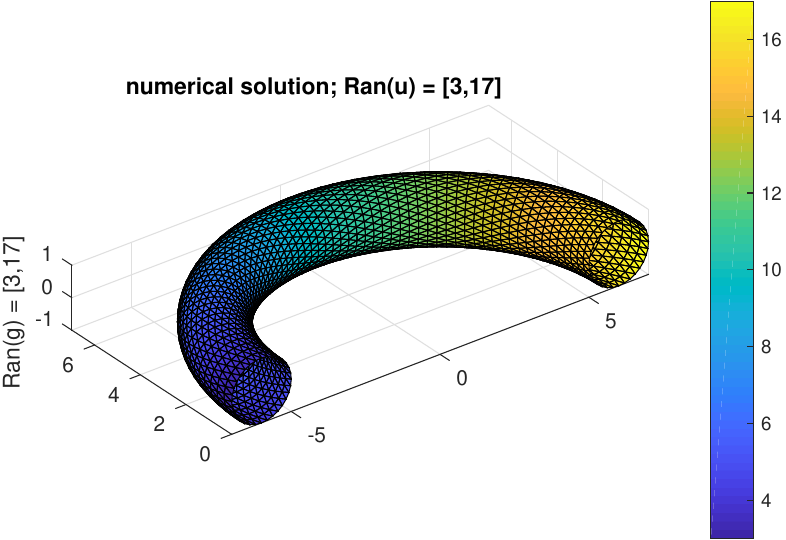} \\
	\includegraphics[width=.7\textwidth,height=0.26\textheight]{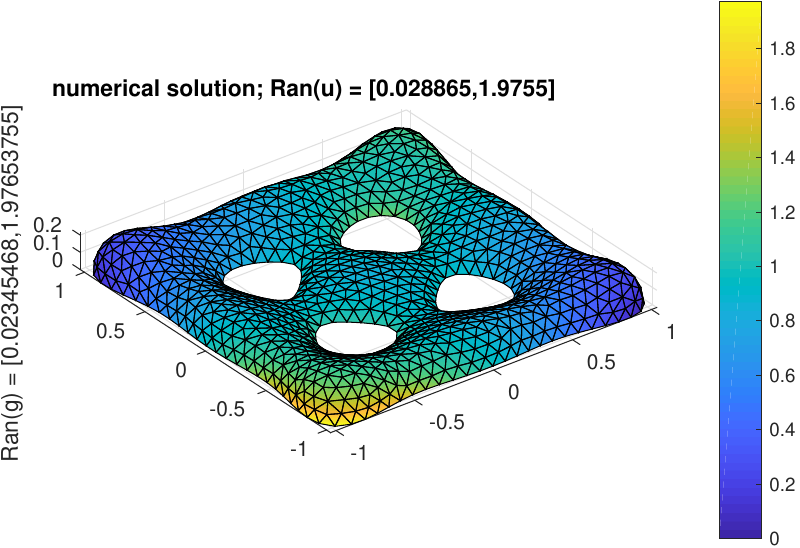} \\
	\caption{Surface meshes and the numerical solutions for the three nonlinear problems (from top to bottom).}
	\label{figure:plots}
\end{figure}

\clearpage

\section*{Acknowledgements}

The work of J. Kar\'atson is supported by the  Hungarian Scientific Research Fund OTKA, No.~K112157 and SNN125119.
The work of B. Kov\'{a}cs is funded by Deutsche Forschungsgemeinschaft, SFB 1173.
A grant from the E\"{o}tv\"{o}s Lor\'{a}nd University let B.K.\ spend a week in Budapest,  allowing personal discussions with J.K.

\end{document}